\documentclass[11pt,reqno]{amsproc}

\usepackage{amsmath,amsfonts,amssymb,amsthm}
\usepackage[abbrev,lite,nobysame]{amsrefs}
\usepackage{mathrsfs}
\usepackage{graphics,graphicx}
\usepackage[usenames,dvipsnames]{color}
\usepackage{times}
\usepackage{bbm}
\usepackage[margin=1in]{geometry}
\usepackage[colorlinks=true, pdfstartview=FitV, linkcolor=blue, citecolor=blue, urlcolor=blue]{hyperref}

\makeatletter

\renewcommand\subsection{\@startsection{subsection}{2}%
\normalparindent{.5\linespacing\@plus.7\linespacing}{-.5em}
{\normalfont\bfseries}}

\renewcommand\subsubsection{\@startsection{subsubsection}{3}%
\normalparindent{.5\linespacing\@plus.7\linespacing}{-.5em}
{\normalfont\bfseries}}

\newcommand{\bullpar}[1]{\vspace{.5em}\noindent $\bullet$ \normalfont {\bfseries #1.}}

\def\@tocline#1#2#3#4#5#6#7{\relax
  \ifnum #1>\c@tocdepth 
  \else
    \par \addpenalty\@secpenalty\addvspace{#2}%
    \begingroup \hyphenpenalty\@M
    \@ifempty{#4}{%
      \@tempdima\csname r@tocindent\number#1\endcsname\relax
    }{%
      \@tempdima#4\relax
    }%
    \parindent\z@ \leftskip#3\relax \advance\leftskip\@tempdima\relax
    \rightskip\@pnumwidth plus4em \parfillskip-\@pnumwidth
    #5\leavevmode\hskip-\@tempdima
      \ifcase #1
       \or\or \hskip 1em \or \hskip 2em \else \hskip 3em \fi%
      #6\nobreak\relax
    \dotfill\hbox to\@pnumwidth{\@tocpagenum{#7}}\par
    \nobreak
    \endgroup
  \fi}
\makeatother

\newtheorem{theorem}{Theorem}
\newtheorem{proposition}{Proposition}[section]
\newtheorem{lemma}[proposition]{Lemma}
\newtheorem{corollary}[proposition]{Corollary}

\theoremstyle{definition}
\newtheorem{definition}[proposition]{Definition}

\numberwithin{equation}{section}

\newcommand{\C}{{\mathbb C}}
\newcommand{\E}{{\mathcal E}}
\newcommand{\cI}{{\mathcal I}}

\newcommand{\G}{{\mathcal{G}}}

\newcommand{\R}{{\mathbb R}}
\newcommand{\Z}{{\mathbb Z}}

\newcommand{\D}{{\Delta}}
\newcommand{\T}{{\mathbb T}}
\newcommand{\Rho}{{\mathrm{P}}}

\renewcommand{\O}{{\Omega}}
\renewcommand{\o}{{\omega}}

\renewcommand{\d}{\mathrm{d}}
\newcommand{\ep}{\varepsilon}
\newcommand\e{{\rm e}}
\renewcommand{\u}{\boldsymbol{u}}

\renewcommand{\b}{\beta}

\newcommand{\p}{\partial}

\renewcommand{\l}{\left}

\renewcommand{\r}{\right}

\def\Re{{\rm Re}}
\def\Im{{\rm Im}}

\usepackage{enumerate}
\usepackage{enumitem}

\begin{document}

\title[Explicit solutions to Euler-Boussinesq in the periodic strip]{Explicit solutions and linear inviscid damping in the Euler-Boussinesq equation near a stratified Couette flow in the periodic strip}

\author[M. Coti Zelati]{Michele Coti Zelati}
\address{Department of Mathematics, Imperial College London, London, SW7 2AZ, UK}
\email{m.coti-zelati@imperial.ac.uk}

\author[M. Nualart]{Marc Nualart}
\address{Department of Mathematics, Imperial College London, London, SW7 2AZ, UK}
\email{m.nualart-batalla20@imperial.ac.uk }



\begin{abstract}
This short note provides explicit solutions to the linearized Boussinesq equations around the stably stratified Couette flow posed on $\T\times\R$. We consider the long-time behavior of such solutions and prove inviscid damping of the perturbed density and velocity field for any positive Richardson number, with optimal rates. The explicit solution is obtained through the limiting absorption principle whereas the inviscid damping is proved using oscillatory integral methods.
\end{abstract}

\maketitle

\setcounter{tocdepth}{1}
\tableofcontents

\section{Introduction}\label{inviscid}
The Euler equations under the Boussinesq approximation
\begin{equation}\label{eq:EBintro}
\begin{aligned}
(\p_t+\tilde\u\cdot\nabla)\tilde\o  &= - \mathfrak{g}\,\p_x\tilde\rho,\\
(\p_t+\tilde\u\cdot\nabla)\tilde\rho&=0, 
\end{aligned}
\end{equation}
models the evolution of an incompressible, non-homogeneous ideal fluid whose velocity field is $\tilde\u=\nabla^\perp\Delta^{-1}\tilde\o$, with associated vorticity $\tilde\o=\nabla^\perp \cdot \tilde\u$ and where the density of the fluid is given by $\tilde\rho$. Here, $\mathfrak{g}$ is the gravity constant.

The physical domain in which we consider the Euler-Boussinesq system \eqref{eq:EBintro} is the periodic strip $\T\times\R$, where  
\begin{equation}\label{eq:StratCouette}
\bar{\u}=(y,0), \quad \bar{\rho}(y)=1-\vartheta y, \quad \p_y p=-\mathfrak{g}\bar{\rho}(y),
\end{equation}
constitutes a steady solution for the equations of motion and represents a stably stratified Couette flow whose density slope is $\vartheta>0$. Our interest lies in describing the linearized long-time dynamics of solutions to \eqref{eq:EBintro} that are near the stationary configuration \eqref{eq:StratCouette}. As such, we consider the perturbed velocity 
$\tilde{\u}=\bar{\u}+\u$ and density profile $\tilde{\rho}=\bar{\rho}+\vartheta\rho$, and define the corresponding vorticity perturbation $\o=\nabla^\perp\cdot \u$. The linearized Euler-Boussinesq system \eqref{eq:EBintro} nearby the stably stratified Couette flow \eqref{eq:StratCouette} then takes the form
\begin{equation}\label{eq:linEulerBouss}
\begin{cases}
\p_t\o + y\p_x\o=-\b^2\p_x\rho \\
\p_t\rho + y\p_x\rho =\p_x\psi,\\
\D\psi=\o,
\end{cases}
\end{equation}
with $\psi$ being the stream-function of the velocity field $\u$ and $\beta=\sqrt{\vartheta \mathfrak{g}} >0$. 
In the periodic setting $x\in\T$, it is advantageous to write 
\begin{equation}\label{xFourier}
    \o(t,x,y)=\sum_{m\in\Z}\o_m(t,y)\e^{imx},\quad \rho(t,x,y)=\sum_{m\in\Z}\rho_m(t,y)\e^{imx}, \quad \psi(t,x,y)=\sum_{m\in\Z}\psi_{m}(t,y)\e^{imx}.
\end{equation}
Thus, \eqref{eq:linEulerBouss} now reads
\begin{align}
    (\p_t +imy) \o_m &= -im\b^2\rho_m \label{eq:lin om}\\
    (\p_t + imy) \rho_m &= im\psi_m, \label{eq: lin rhom}
\end{align}
where further
\begin{equation}\label{eq:defDm}
    \begin{cases}
        \D_m\psi_m = \o_m,  \\
        \lim_{|y|\to  \infty}\psi_m=0,  
    \end{cases} \qquad \D_m:= \p_y^2 - m^2,
\end{equation}
for all $m\in\Z$. Our first result shows that \eqref{eq:lin om}-\eqref{eq: lin rhom}, and thus \eqref{eq:linEulerBouss} through \eqref{xFourier} can be solved explicitly in the physical space $y\in\R$ as truncated convolutions of oscillating Whittaker functions against suitable combinations of the initial data. The Whittaker functions $W_{0,\gamma}$ with $\gamma\in\C$ satisfy (see \cite{NIST,Whittaker03}) 
\begin{equation}\label{eq:W0gamma}
    \p_\zeta^2W_{0,\gamma} + \l( -\frac14 + \frac{1/4 - \gamma^2}{\zeta^2}\r)W_{0,\gamma}=0, \qquad W_{0,\gamma}(\zeta)\sim \e^{-\frac12\zeta} \quad\text{ as }\quad \zeta\rightarrow\infty,
\end{equation}
and constitute the main ingredient in the construction of the explicit solutions.

\begin{theorem}\label{thm: main solution formula}
Let $\b> 0$. Given initial conditions $(\o^0, \rho^0)$ such that
\begin{equation}\label{eq:zeroxave}
\int_\T \o^0(x,y)\d x = \int_\T\rho^0(x,y)\d x =0,
\end{equation}
the unique solution to \eqref{eq:linEulerBouss} is given through \eqref{xFourier} and \eqref{eq:defDm} by $\psi_0=\rho_0\equiv 0$,
\begin{equation}\label{eq:psim}
\begin{aligned}
\psi_m(t,y)&=\frac{\e^{-imyt}}{2|m|\pi}\cos(\gamma\pi) \l(\int_0^\infty \e^{im\eta t}W(\eta) \int_0^\infty W(\xi)G_{m}(\eta,\xi,y)\d \xi\d \eta \r. \\
&\qquad\qquad\qquad\qquad\qquad \l. -\int_0^\infty \e^{-im\eta t}W(\eta)\int_0^\infty W(\xi)G_{m}(-\eta,-\xi,y)\d \xi\d \eta \r)
\end{aligned}
\end{equation}
and 
\begin{equation}\label{eq:rhom}
\begin{aligned}
\rho_m(t,y)&=\frac{\e^{-imyt}}{2|m|\pi}\cos(\gamma\pi) \l(\int_0^\infty \e^{im\eta t}\frac{W(\eta)}{\eta} \int_0^\infty W(\xi)G_{m}(\eta,\xi,y)\d \xi\d \eta \r. \\
&\qquad\qquad\qquad\qquad\qquad \l. +\int_0^\infty \e^{-im\eta t}\frac{W(\eta)}{\eta}\int_0^\infty W(\xi)G_{m}(-\eta,-\xi,y)\d \xi\d \eta \r),
\end{aligned}
\end{equation}
for all $m\neq 0$. Here, for $\gamma:=\sqrt{1/4-\b^2}$ we denote $W(\cdot)=W_{0,\gamma}(2|m|\cdot)$, where $W_{0,\gamma}$ satisfies \eqref{eq:W0gamma}. Further, 
\begin{equation*}
G_m(\eta,\xi,y)=\D_m\l(\rho_m^0(\xi+y-\eta)-\frac{1}{\b^2}\xi \o_m^0(\xi+y-\eta)\r).
\end{equation*}
\end{theorem}

The zero modes in $x\in\T$ of initial data evolving according to \eqref{eq:linEulerBouss} are constants of motion, so that \eqref{eq:zeroxave} does not actually constitute a restriction on the initial data. The expressions \eqref{eq:psim} and \eqref{eq:rhom} give rise to real-valued solutions $\psi$, $\o$ and $\rho$ via \eqref{xFourier} due to the fact that $\psi_m = \overline{\psi_{-m}}$ and $\rho_m$ = $\overline{\rho_{-m}}$, which are straightforward consequences of $\o^0$ and $\rho^0$ being real-valued. In particular, we shall assume without loss of generality throughout the article that $m\geq 1$.

Our second result consists in the derivation of sharp decay estimates, which quantify the phenomenon of inviscid damping.

\begin{theorem}\label{thm:mainstrip}
Let $\b>0$ and assume that the initial data $(\o^0,\rho^0)$ to \eqref{eq:linEulerBouss} satisfies \eqref{eq:zeroxave}. Let $\u=(u^x,u^y)=\nabla^\perp\psi=(-\p_y\psi,\p_x\psi)$ be the corresponding velocity field. We have the following estimates. 
\begin{itemize}
\item If $\b^2\neq1/4$, let $\mu=\Re\sqrt{1/4-\b^2}$ and $\nu=\Im\sqrt{1/4-\b^2}$. Then,
 \begin{align}
\Vert u^x(t) \Vert_{L^2}&\lesssim\frac{1}{ t^{\frac12-\mu}}\l( \Vert \rho^0 \Vert_{L^2_xH^3_y} + \Vert \o^0 \Vert_{L^2_xH^3_y}\r), \label{eq:decayvx} \\
\Vert u^y(t) \Vert_{L^2}&\lesssim \frac{1}{ t^{\frac32-\mu}}\l( \Vert \rho^0 \Vert_{L^2_xH^4_y} + \Vert \o^0 \Vert_{L^2_xH^4_y}\r), \label{eq:decayvy}\\
\Vert \rho(t) \Vert_{L^2}&\lesssim \frac{1}{t^{\frac12-\mu}}\l( \Vert \rho^0 \Vert_{L^2_xH^3_y} + \Vert \o^0 \Vert_{L^2_xH^3_y}\r), \label{eq:decayrho} 
\end{align}
for every $t\geq 1$.
\item If $\b^2=1/4$, then 
\begin{align}
\Vert u^x(t) \Vert_{L^2}&\lesssim \frac{1+\log(t)}{ t^\frac12}\l( \Vert \rho^0 \Vert_{L^2_xH^3_y} + \Vert \o^0 \Vert_{L^2_xH^3_y}\r), \label{eq:decayvxlog}\\
\Vert u^y(t) \Vert_{L^2}&\lesssim \frac{1+\log(t)}{ t^\frac32}\l( \Vert \rho^0 \Vert_{L^2_xH^4_y} + \Vert \o^0 \Vert_{L^2_xH^4_y}\r), \label{eq:decayvylog}\\
\Vert \rho(t) \Vert_{L^2}&\lesssim\frac{1+\log(t)}{t^\frac12}\l( \Vert \rho^0 \Vert_{L^2_xH^3_y} + \Vert \o^0 \Vert_{L^2_xH^3_y}\r),   \label{eq:decayrholog} 
\end{align}
for every $t\geq 1$.
\end{itemize}
\end{theorem}

The inviscid damping estimates \eqref{eq:decayvx}-\eqref{eq:decayrholog} describe the long-time dynamics of solutions to \eqref{eq:linEulerBouss} and show the linear asymptotic stability of the stratified Couette configuration \eqref{eq:StratCouette} for the Euler-Boussinesq system \eqref{eq:EBintro}. The decay is produced by two phenomena. Firstly, there is \emph{mixing} due to the background Couette flow and secondly there is \emph{stratification} due to the background density. The effect of mixing has been thoroughly studied in the homogeneous Euler equations both at the linear \cites{WZZ18,WZZ19,Zillinger16,JiaGev20} and non-linear level \cites{BM15,IJnon20,MZ20}. 

Estimates analogous to those of Theorem \ref{thm:mainstrip} have been already obtained in \cite{YL18} using an explicit formula for solutions on the Fourier side (inspired by an early work of Hartman \cite{Hartman} in 1975), and in \cite{BCZD22} via an energy method. Our approach is rather based on a stationary-phase type argument, exploiting the explicit solutions of Theorem \ref{thm: main solution formula} in physical space and obtaining decay rates related to the regularity (and more precisely on the asymptotic expansion) of the Whittaker functions about the origin. While these formulae do not produce a new result in the periodic strip $\T\times\R$, our method allows to treat the physically relevant case of the periodic channel $\T\times [0,1]$, see \cite{CZN23chan}, and it is therefore more robust in this sense. In \cite{CZN23chan} explicit solutions are not available, however one similarly can write solutions to \eqref{eq:linEulerBouss} through oscillatory integrals now involving a limiting absorption principle in which the regularity of the limiting functions (and thus the gained time-decay via stationary-phase arguments) is related to that of the Whittaker functions.

\subsection{Notation and assumptions}
Throughout the article, we assume that $\b>0$ and $m\geq 1$. To quantify the regularity of the initial data, for $j\geq 0$ we introduce
\begin{equation*}
    Q_{j,m}=\Vert \rho_m^0 \Vert_{H^{2+j}_y(\R)} + \Vert \o_m^0 \Vert_{H^{2+j}_y(\R)}.
\end{equation*}
As usual, we say that $A\lesssim B$ when there exists $C>0$ such that $A\leq CB$.

\subsection{Plan of the article}
In Section \ref{sec: proof thm sol} we prove Theorem \ref{thm: main solution formula} and in Section \ref{sec:heuristics} we provide an heuristic explanation for the form of the solutions \eqref{eq:psim}, \eqref{eq:rhom}. Section \ref{sec: time decay estimates} is devoted to the proof of Theorem \ref{thm:mainstrip}. In the Appendix \ref{app: Whittaker} we provide the main asymptotic  expansions for the Whittaker functions that are used to establish Theorem \ref{thm:mainstrip}.

\section{Proof of Theorem \ref{thm: main solution formula}}\label{sec: proof thm sol}
The proof consists on showing that $\psi_m$, $\rho_m$ and $\o_m$ given by \eqref{eq:psim}, \eqref{eq:rhom} and \eqref{eq:defDm} respectively, satisfy the linearized Euler-Boussinesq equations \eqref{eq:lin om}-\eqref{eq: lin rhom}. According to \eqref{eq:psim}, \eqref{eq:rhom} and \eqref{eq:defDm}, we write 
\begin{equation}\label{eq:defPsiRho}
   \psi_m(t,y)=\e^{-imyt}\Psi_m(t,y), \qquad \rho_m(t,y)=\e^{-imty}\Rho_m(t,y), \qquad  \o_m(t,y)=\e^{-imyt}\O_m(t,y),
\end{equation}
where
\begin{equation}\label{eq:defOm}
    \O_m :=-m^2t^2\Psi_m -2imt\p_y\Psi_m +\D_m\Psi_m .
\end{equation}
Now, with this formulation we must check that $\O_m,\Psi_m$ satisfy
\begin{align}
\p_t\O_m&=-im\b^2\Rho_m \label{eq:Om}\\
\p_t\Rho_m&=im\Psi_m. \label{eq:Rhom}
\end{align}
Clearly, \eqref{eq:Rhom} follows directly from \eqref{eq:psim}, \eqref{eq:rhom} and \eqref{eq:defPsiRho}. To show \eqref{eq:Om}, we first notice that
\begin{equation*}
\p_t\O_m=-2m^2t\Psi_m -m^2t^2\p_t\Psi_m -2im\p_y\Psi_m -2imt\p_t\p_y\Psi_{m}+ \D_m\p_t\Psi_m,
\end{equation*}
where, from \eqref{eq:psim} and \eqref{eq:defPsiRho}, we have that
\begin{equation*}
\begin{aligned}
-2m^2t\Psi_m&=-2m^2t\frac{\cos(\gamma\pi)}{2m\pi} \l(\int_0^\infty \frac{1}{imt}\p_\eta    \l(\e^{im\eta t}\r)W(\eta ) \int_0^\infty W(\xi)G_m(\eta,\xi,y)\d \xi\d \eta  \r. \\
&\qquad\qquad\qquad\qquad\qquad \l. +\int_0^\infty \frac{1}{imt}\p_\eta   \l( \e^{-im\eta t}\r)W(\eta )\int_0^\infty W(\xi)G_m(-\eta,-\xi,y)\d \xi\d \eta  \r) \\
&=-2mi\frac{\cos(\gamma\pi)}{2m\pi} \l(\int_0^\infty \e^{im\eta t}\p_\eta   \l(W(\eta ) \int_0^\infty W(\xi)G_m(\eta,\xi,y)\d \xi\r)\d \eta  \r. \\
&\qquad\qquad\qquad\qquad\qquad \l. +\int_0^\infty \e^{-im\eta t}\p_\eta   \l(W(\eta )\int_0^\infty W(\xi)G_m(-\eta,-\xi,y)\d \xi\r)\d \eta  \r),
\end{aligned}
\end{equation*}
while
\begin{equation*}
\begin{aligned}
-m^2t^2\p_t\Psi_m&=-m^2t^2\frac{\cos(\gamma\pi)}{2m\pi} \l(\int_0^\infty \e^{im\eta t}im\eta W(\eta ) \int_0^\infty W(\xi)G_m(\eta,\xi,y)\d \xi\d \eta  \r. \\
&\qquad\qquad\qquad\qquad\qquad \l. +\int_0^\infty \e^{-im\eta t}im\eta W(\eta )\int_0^\infty W(\xi)G_m(-\eta,-\xi,y)\d \xi\d \eta  \r) \\
&=im\frac{\cos(\gamma\pi)}{2m\pi} \l(\int_0^\infty  \e^{im\eta t} \p_\eta   ^2\l( \eta W(\eta) \int_0^\infty W(\xi)G_m(\eta,\xi,y)\d \xi\r)\d \eta  \r. \\
&\qquad\qquad\qquad\qquad\qquad \l. +\int_0^\infty  \e^{-im\eta t}\p_\eta   ^2\l(\eta W(\eta)\int_0^\infty W(\xi)G_m(-\eta,-\xi,y)\d \xi\r)\d \eta  \r)
\end{aligned}
\end{equation*}
and 
\begin{equation*}
\begin{aligned}
-2im\p_y\Psi_m&=-2im\frac{\cos(\gamma\pi)}{2m\pi} \l(\int_0^\infty \e^{im\eta t}W(\eta ) \int_0^\infty W(\xi)\p_yG_m(\eta,\xi,y)\d \xi\d \eta  \r. \\
&\qquad\qquad\qquad\qquad\qquad \l. -\int_0^\infty \e^{-im\eta t}W(\eta )\int_0^\infty W(\xi)\p_yG_m(-\eta,-\xi,y)\d \xi\d \eta  \r),
\end{aligned}
\end{equation*}
with also
\begin{equation*}
\begin{aligned}
-2imt\p_t\p_y\Psi_m&=-2imt\frac{\cos(\gamma\pi)}{2m\pi} \l(\int_0^\infty \e^{im\eta t}im\eta W(\eta ) \int_0^\infty W(\xi)\p_yG_m(\eta,\xi,y)\d \xi\d \eta  \r. \\
&\qquad\qquad\qquad\qquad\qquad \l. +\int_0^\infty \e^{-im\eta t}im\eta W(\eta )\int_0^\infty W(\xi)\p_yG_m(-\eta,-\xi,y)\d \xi\d \eta  \r) \\
&=2im\frac{\cos(\gamma\pi)}{2m\pi} \l(\int_0^\infty  \e^{im\eta t} \p_\eta   \l( \eta W(\eta) \int_0^\infty W(\xi)\p_yG_m(\eta,\xi,y)\d \xi\r)\d \eta  \r. \\
&\qquad\qquad\qquad\qquad\qquad \l. -\int_0^\infty  \e^{-im\eta t}\p_\eta   \l(\eta W(\eta)\int_0^\infty W(\xi)\p_yG_m(-\eta,-\xi,y)\d \xi\r)\d \eta  \r)
\end{aligned}
\end{equation*}
and finally
\begin{equation*}
\begin{aligned}
\D_m\p_t\Psi_m&=im\frac{\cos(\gamma\pi)}{2m\pi} \l(\int_0^\infty \e^{im\eta t}\eta  W(\eta ) \int_0^\infty W(\xi)\D_mG_m(\eta,\xi,y)\d \xi\d \eta  \r. \\
&\qquad\qquad\qquad\qquad\qquad \l. +\int_0^\infty \e^{-im\eta t}\eta W(\eta)\int_0^\infty W(\xi)\D_mG_m(-\eta,-\xi,y)\d \xi\d \eta  \r).
\end{aligned}
\end{equation*}
Therefore, it is straightforward to see that
\begin{equation*}
\begin{aligned}
\p_t\O_m=im\frac{\cos(\gamma\pi)}{2m\pi}\l( \int_0^\infty \e^{im\eta t}\O_m^{(+)}(\eta )\d \eta  + \int_0^\infty \e^{im\eta t}\O_m^{(-)}(\eta )\d \eta \r),
\end{aligned}
\end{equation*}
where
\begin{equation*}
\begin{aligned}
\O_m^{(+)}(\eta )&=-2\p_\eta    \l( W(\eta ) \int_0^\infty W(\xi)G_m(\eta,\xi,y)\d \xi \r) + \p_\eta   ^2 \l( \eta W(\eta) \int_0^\infty W(\xi)G_m(\eta,\xi,y)\d \xi \r)
\\ &\quad- 2 W(\eta ) \int_0^\infty W(\xi)\p_yG_m(\eta,\xi,y)\d \xi + 2\p_\eta   \l( \eta W(\eta) \int_0^\infty W(\xi)\p_yG_m(\eta,\xi,y)\d \xi \r) \\ 
&\quad+ \eta W(\eta) \int_0^\infty W(\xi)\D_mG_m(\eta,\xi,y)\d \xi 
\end{aligned}
\end{equation*}
and similarly
\begin{equation*}
\begin{aligned}
\O_m^{(-)}(\eta )&=-2\p_\eta    \l( W(\eta ) \int_0^\infty W(\xi)G_m(-\eta,-\xi,y)\d \xi \r) + \p_\eta   ^2 \l( \eta W(\eta) \int_0^\infty W(\xi)G_m(-\eta,-\xi,y)\d \xi \r)
\\ &\quad+ 2 W(\eta ) \int_0^\infty W(\xi)\p_yG_m(-\eta,-\xi,y)\d \xi - 2\p_\eta   \l( \eta W(\eta) \int_0^\infty W(\xi)\p_yG_m(-\eta,-\xi,y)\d \xi \r) \\ 
&\quad+ \eta W(\eta) \int_0^\infty W(\xi)\D_mG_m(-\eta,-\xi,y)\d \xi.
\end{aligned}
\end{equation*}
Now, note that 
\begin{equation*}
\begin{aligned}
\O_m^{(+)}(\eta )&=\eta \p_\eta   ^2 \l( W(\eta ) \int_0^\infty W(\xi)G_m(\eta,\xi,y)\d \xi \r)
+ 2\eta\p_\eta   \l( W(\eta ) \int_0^\infty W(\xi)\p_yG_m(\eta,\xi,y)\d \xi \r) \\ 
&\quad+ \eta W(\eta) \int_0^\infty W(\xi)\D_mG_m(\eta,\xi,y)\d \xi  \\
&=\eta  W''(\eta ) \int_0^\infty W(\xi)G_m(\eta,\xi,y)\d \xi + 2 \eta W'(\eta ) \int_0^\infty W(\xi)\p_\eta  G_m(\eta,\xi,y)\d \xi \\
&\quad+ \eta  W(\eta ) \int_0^\infty W(\xi)\p_\eta   ^2 G_m(\eta,\xi,y)\d \xi + 2\eta W'(\eta ) \int_0^\infty W(\xi)\p_yG_m(\eta,\xi,y)\d \xi \\
&\quad +2\eta W(\eta ) \int_0^\infty W(\xi)\p_\eta   \p_yG_m(\eta,\xi,y)\d \xi + \eta W(\eta) \int_0^\infty W(\xi)\D_mG_m(\eta,\xi,y)\d \xi
\end{aligned}
\end{equation*}
and further observe that $\l(\p_y+\p_\eta  \r) G_m(\eta ,\xi,y)=0$, from which we deduce that 
\begin{equation*}
\l(\p_\eta   ^2 +2\p_\eta   \p_y+\p_y^2\r) G_m(\eta,\xi,y)= (\p_\eta   +\p_y)^2\G_m(\eta ,\xi,y)=0
\end{equation*}
and 
\begin{equation*}
  \O_m^{(+)}(\eta )=\eta \D_mW(\eta ) \int_0^\infty W(\xi)G_m(\eta,\xi,y)\d \xi  
\end{equation*}
Now, using \eqref{eq:W0gamma}, we see that
\begin{equation*}
    \D_m W(\zeta) +\b^2\frac{W(\zeta)}{\zeta^2}=0,
\end{equation*}
and thus we can write
\begin{equation*}
\O_m^{(+)}(\eta )= -\b^2\frac{W(\eta )}{\eta } \int_0^\infty W(\xi)G_m(\eta,\xi,y)\d \xi
\end{equation*}
and similarly for $\O_m^{(-)}(\eta )$. We finish by assembling and recognising $\Rho_m$,
\begin{equation*}
\begin{aligned}
\p_t\O_m&=im\frac{\cos(\gamma\pi)}{2m\pi}\l( -\int_0^{\infty} \e^{im\eta t}\b^2\frac{W(\eta )}{\eta } \int_0^\infty W(\xi)G_m(\eta,\xi,y)\d \xi \d \eta  \r. \\
&\qquad\qquad\qquad\qquad\qquad \l. -\int_0^{\infty} \e^{-im\eta t}\b^2\frac{W(\eta )}{\eta } \int_0^\infty W(\xi)G_m(-\eta,-\xi,y)\d \xi \d \eta \r) \\
&=-im\b^2 \Rho_m
\end{aligned}
\end{equation*}
With this, the proof is concluded.

\section{Heuristics for the explicit solutions}\label{sec:heuristics}
The presence of the Whittaker functions in \eqref{eq:psim} and \eqref{eq:rhom} is key in the proof of Theorem \ref{thm: main solution formula}, since they relate $\O_m^{(\pm)}$ to $\Rho_m$ due to \eqref{eq:W0gamma}. In fact, this is essentially the main reason why \eqref{eq:psim} and \eqref{eq:rhom} provide solutions to \eqref{eq:linEulerBouss}. However, the proof of Theorem \ref{thm: main solution formula} does not explain why these Whittaker functions arise in \eqref{eq:psim} and \eqref{eq:rhom}. This is precisely the purpose of this section, which sets the framework for obtaining \eqref{eq:psim} and \eqref{eq:rhom} via the method of the limiting absorption principle.

\subsection{Generalized stream-functions and densities}
Writing \eqref{eq:lin om}-\eqref{eq: lin rhom} in the compact  stream-function formulation 
\begin{equation}
\p_t \begin{pmatrix} \psi_m \\ \rho_m\end{pmatrix}+imL_m\begin{pmatrix}
\psi_m \\ \rho_m
\end{pmatrix}=0,
\end{equation}
we directly obtain its solution as
\begin{equation}
\begin{pmatrix} \psi_m \\ \rho_m\end{pmatrix}=\e^{-imL_mt}\begin{pmatrix}
\psi_m^0 \\ \rho_m^0
\end{pmatrix},
\end{equation}
where $L_m$ is the linear operator defined by
\begin{equation}\label{eq:linOP}
L_m =\begin{pmatrix}
\D_m^{-1}(y\D_m)  & \b^2\D_m^{-1} \\
-1 & y
\end{pmatrix}.
\end{equation}
Using Dunford's formula \cite{Engel-Nagel, Taylor-11}, we have that
\begin{equation}\label{eq:Dunford}
\begin{pmatrix}
\psi_m(t,y) \\ \rho_m(t,y)
\end{pmatrix} = \frac{1}{2\pi i} \int_{\p \mathrm{D}}\e^{-imct} (c-L_m)^{-1}\begin{pmatrix}
\psi_m^0(y) \\ \rho_m^0(y)
\end{pmatrix} \,\d c, 
\end{equation}
where here $\mathrm{D}$ is any domain containing the spectrum $\sigma(L_m)$. On the periodic strip, the spectrum $\sigma(L_m)$ is continuous and consists on the real line $\R$. Hence, we can reduce the contour of integration to  
\begin{align}\label{eq:LAPpsirho}
\begin{pmatrix}
\psi_m(t,y) \\ \rho_m(t,y)
\end{pmatrix} 
&=\frac{1}{2\pi i }\lim_{\ep\rightarrow 0}\int_{-\infty}^{+\infty} \e^{-imy_0t}\l[(-y_0-i\ep+L_m)^{-1}-(-y_0+i\ep+L_m)^{-1}\r]
\begin{pmatrix}
\psi_m^0 \\ \rho_m^0 
\end{pmatrix}\, \d y_0.
\end{align}
For $\ep>0$, we denote
\begin{equation*}
\begin{pmatrix}
\psi^{\pm}_{m,\ep}(y,y_0) \\ \rho^\pm_{m,\ep}(y,y_0)
\end{pmatrix}:=\l( -y_0\pm i\ep+L_m\r)^{-1}\begin{pmatrix}
\psi_m^0(y) \\ \rho_m^0(y)
\end{pmatrix}
\end{equation*}
and obtain the coupled system of equations for the generalized stream-functions $\psi_{m,\ep}^\pm$ and generalized densities $\rho_{m,\ep}^\pm$
\begin{equation*}
\begin{aligned}
\o_m^0(y)&=(y-y_0\pm i\ep)\D_m\psi^\pm_{m,\ep}(y,y_0)+\b^2 \rho^\pm_{m,\ep}(y,y_0), \\
\rho_m^0(y)&=(y-y_0\pm i\ep)\rho^\pm_{m,\ep}(y,y_0) -\psi^\pm_{m,\ep}(y,y_0).
\end{aligned}
\end{equation*}
We first solve for the generalized densities
\begin{equation}\label{eq rho m ep}
\rho^\pm_{m,\ep}(y,y_0)=\frac{1}{y-y_0\pm i\ep}\l( \rho_m^0(y)+\psi^\pm_{m,\ep}(y,y_0)\r)
\end{equation}
and from there we obtain the following inhomogeneous \emph{Taylor-Goldstein equation} for the generalized stream-functions $\psi^\pm_{m,\ep}$,
\begin{equation}\tag{TG}\label{eq:TG}
\D_m\psi^{\pm}_{m,\ep}+\b^2\frac{\psi^\pm_{m,\ep}}{(y-y_0\pm i\ep)^2}=\frac{\o_m^0}{y-y_0\pm i\ep}-\b^2\frac{\rho_m^0}{(y-y_0\pm i\ep)^2},
\end{equation}
along with the vanishing of $\psi_{m,\ep}^\pm$ at infinity.

\subsection{Explicit solutions for the generalized stream-functions and densities}
The Taylor-Goldstein equation \eqref{eq:TG} admits a fairly explicit Green's function given by
\begin{equation*}\label{eq: Greens}
\G_{m,\ep}^\pm(y,y_0,z) =-\frac{1}{2m}\begin{cases} 
W(y_0-z\mp i\ep)W(y-y_0\pm i\ep), & y\geq z,  \\
W(z-y_0\pm i\ep)W(y_0-y\mp i\ep), & y\leq z,
\end{cases}
\end{equation*}
where we recall $W(\zeta)=W_{0,\gamma}(2m\zeta)$ for $\gamma^2=\frac14 - \b^2$ and it is such that 
\begin{equation}\label{eq:Whittakereqn}
\p^2_{\zeta} W+\left(-{\frac  {1}{4}} +{\frac  {1/4- (1/4-\b^2)}{4m^2\zeta^{2}}}\right)W=0, \quad W(\zeta)\sim \e^{-m\zeta}, \text{ as } \zeta\rightarrow \infty,
\end{equation}
since the Whittaker function $W_{0,\gamma}$ satisfies \eqref{eq:W0gamma}. To obtain suitable formulas for the generalized stream-functions and densities, define 
\begin{equation*}\label{eq:Fdata}
H_{m,\ep}^\pm(z,y_0):=\D_m\rho_m^0(z)-\frac{1}{\b^2}\D_m\big((z-y_0\pm i\ep)\o_m^0(z)\big).
\end{equation*}
and assume that the initial data vanish at infinity. Then, the solution $\psi^\pm_{m,\ep}(y,y_0)$  to \eqref{eq:TG}
is
\begin{equation}\label{def psi 2}
\begin{aligned}
\psi^\pm_{m,\ep}(y,y_0)&= \frac{1}{\b^2}(y-y_0\pm i\ep)\o_m^0(y) -\rho_m^0(y) + \int_{-\infty}^{+\infty} \G^\pm_{m,\ep}(y,y_0,z) H_{m,\ep}^\pm(z,y_0) \d z.
\end{aligned}
\end{equation}
and the generalized density is given by
\begin{equation}\label{def rho 2}
\rho_{m,\ep}^\pm(y,y_0)=\frac{1}{\b^2}\o_m^0(y) +\frac{1}{y-y_0\pm i\ep}\int_{-\infty}^{+\infty} {\G}_{m,\ep}^\pm(y,y_0,z) H_{m,\ep}^\pm(z,y_0)\d z.
\end{equation}

\subsection{The limiting absorption principle}
With \eqref{def psi 2} and \eqref{def rho 2} at hand, one may precisely compute the limiting absorption principle, that is, we may precisely compute \eqref{eq:LAPpsirho}
For instance, to obtain $\psi_m(t,y)$ one may compute
\begin{equation*}
\lim_{\ep\rightarrow0}\int_\R \e^{-imy_0t}\int_\R\l( \G_{m,\ep}^-(y,y_0,z)-\G_{m,\ep}^+(y,y_0,z)\r)  H_{m,\ep}^-(z,y_0) \d z \d y_0 
\end{equation*}
and note that with the change of variables $\xi=z-y_0$ and $\eta =y-y_0$, 
\begin{align*}
\int_\R \G_m^\pm(y,z)H_{m,\ep}^-(z,y_0)\d z&=-\frac{1}{2m}W(y-y_0\pm i\ep)\int_{-\infty}^y W(y_0-z\mp i\ep)H_{m,\ep}^-(z,y_0)\d z \\
&\quad-\frac{1}{2m}W(y_0-y\mp i\ep)\int_y^\infty W(z-y_0\pm i\ep)H_{m,\ep}^-(z,y_0)\d z \\
&=-\frac{1}{2m}W(\eta \pm i\ep))\int_{-\infty}^\eta  W(-\xi\mp i\ep))H_{m,\ep}^-(\xi+y-\eta ,y-\eta )\d \xi \\
&\quad-\frac{1}{2m}W(-\eta \mp i\ep)\int_\eta ^\infty W(\xi\pm i\ep)H_{m,\ep}^-(\xi+y-\eta ,y-\eta )\d \xi.
\end{align*}
Setting $G_{m,\ep}^\pm(\eta ,\xi,y):=H_{m,\ep}^\pm(\xi+y-\eta ,y-\eta )$ and $G_m:=\lim_{\ep\rightarrow 0} G_{m,\ep}^\pm=G_{m,0}^\pm$, we find that
\begin{align*}
\int_\R \l(\G_m^-(y,z)-\G_m^+(y,z)\r)&H_{m,\ep}^-(z,y_0)\d z \\
&=\frac{1}{2m}\Big( W(\eta +i\ep)-W(\eta -i\ep)\Big)\int_{-\infty}^\eta  W(-\xi-i\ep)G_{m,\ep}^-(\eta ,\xi,y)\d \xi \\
&\quad+\frac{1}{2m} W(\eta -i\ep)\int_{-\infty}^\eta  \Big( W(-\xi-i\ep)-W(-\xi+i\ep)\Big)G_{m,\ep}^-(\eta ,\xi,y)\d \xi \\
&\quad+ \frac{1}{2m}\Big( W(-\eta -i\ep)-W(-\eta +i\ep)\Big)\int_\eta ^\infty W(\xi+i\ep)G_{m,\ep}^-(\eta ,\xi,y)\d \xi \\
&\quad+\frac{1}{2m}W(-\eta +i\ep)\int_\eta ^\infty \Big( W(\xi+i\ep)-W(\xi-i\ep)\Big) G_{m,\ep}^-(\eta ,\xi,y)\d \xi.
\end{align*}
Taking the limit as $\ep$ vanishes is not trivial. Indeed, $W$ has a branch cut in the negative real axis, see Appendix \ref{app: Whittaker}, and is thus not continuous there. For this reason, we need the analytic continuation of $W$, recorded in the following lemma, whose proof is postponed to Appendix \ref{app: Whittaker}.

\begin{lemma}[Analytic continuation]\label{analytic continuation lemma}
Let $\eta \geq0$ and $0<\ep<1$. Then,
\begin{equation*}
\lim_{\ep\rightarrow0}W(-\eta +i\ep)-W(-\eta -i\ep)=2i\cos(\gamma\pi)W(\eta).
\end{equation*}
\end{lemma}

The whole limiting procedure can be carried out rigorously and produces the explicit formulas exhibited in Theorem \ref{thm: main solution formula}. However, for the sake of brevity, we opted for showing the validity of the explicit formulas by checking they satisfy the linearized system of equations. When the equations \eqref{eq:linEulerBouss} are posed in $\T\times[0,1]$, the limiting procedure becomes much more complicated. Nevertheless, it is still possible to obtain asymptotic expansions on the resulting stream-function and density near the critical layer that  capture the same nature of the explicit formulas \eqref{eq:psim} and \eqref{eq:rhom} for the spatial setting $\T\times\R$, we refer the reader to \cite{CZN23chan}.

\section{Proof of Theorem \ref{thm:mainstrip}}\label{sec: time decay estimates}
In this section we obtain the point-wise decay rates in time for the stream function $\psi_m(t,y)$ and the density $\rho_m(t,y)$. These will be obtained as direct consequence of the following lemma, which concerns the time decay of general oscillatory integrals.
Before stating it, we introduce the following spaces of functions.
\begin{definition}
For $\delta_0>0$ we define
\begin{equation*}
X:= \l\lbrace f:[0,\infty)\times\R \rightarrow \R \text{ such that } \Vert f \Vert_{X}:= \Vert f \Vert_{L^\infty_\eta  \l( 0,\delta_0;L^2_y(\R)\r)} < \infty \r\rbrace
\end{equation*}
and also
\begin{equation*}
Y:= \l\lbrace f:[0,\infty)\times\R \rightarrow \R \text{ such that } \Vert f \Vert_{Y}:= \Vert f \Vert_{L^1_\eta  \l(\delta_0,\infty;L^2_y(\R)\r)} < \infty \r\rbrace.
\end{equation*}
\end{definition}

\begin{lemma}\label{lemma: generic decay-by-regularity}
Let $0<\alpha<1$, $\delta_0=\frac{1}{2m}$ and $t\geq1$. Let $F=F(\eta ,y):(0,\infty)\times\R\rightarrow \R$ be such that $\Vert F(\eta ,\cdot) \Vert_{L^2_y(\R)}$ vanishes as $\eta \rightarrow+\infty$ and $\p_\eta   F\in Y$. We have the following.
\begin{enumerate} [label=(\roman*), ref=(\roman*)]
\item\label{item:i}Assume $F$ admits the decomposition $F(\eta ,y)=\eta ^{-\alpha}E_1(\eta ,y)$, for some $E_1\in X$ and $\p_\eta   F(\eta ,y)=\eta ^{-\alpha-1}E_2(\eta ,y)$ for some $E_2\in X$. Then, 
\begin{equation*}
\l\Vert \int_0^\infty \e^{im\eta t}F(\eta ,y)\d \eta \r\Vert_{L^2_y(\R)} \lesssim \frac{1}{(mt)^{1-\alpha}} \l( \Vert E_1 \Vert_X + \Vert E_2 \Vert_X\r)  + \frac{1}{mt}\Vert \p_\eta   F \Vert_Y.
\end{equation*}
\item\label{item:ii} Assume $F$ admits the decomposition  $F(\eta ,y)=\eta ^{-\alpha}\l( E_{1,1}(\eta ,y) + \log(\eta )E_{1,2}(\eta ,y)\r)$ and \\ $\p_\eta   F(\eta ,y)=\eta ^{-\alpha-1}(E_{2,1}(\eta ,y)+\log(\eta )E_{2,2}(\eta ,y))$ for some $E_{i,j}\in X$, with $i,j\in \l\lbrace 1,2 \r\rbrace$.  Then,
\begin{equation*}
\l\Vert \int_0^\infty \e^{im\eta t}F(\eta ,y)\d \eta \r\Vert_{L^2_y(\R)} \lesssim  \frac{1+\log(mt)}{(mt)^{1-\alpha}} \sum_{i,j\in \l\lbrace 1,2\r\rbrace} \Vert E_{i,j} \Vert_X + \frac{1}{mt}\Vert \p_\eta   F\Vert_Y.
\end{equation*}
\end{enumerate}
\end{lemma}

\begin{proof}
Let $\delta\in (0,\delta_0)$ and set
\begin{equation*}
\begin{aligned}
\cI(y)=\int_0^\infty \e^{im\eta t}F(\eta ,y)\d \eta  &= \l( \int_0^\delta + \int_\delta^\infty \r) \e^{im\eta t}F(\eta ,y)\d \eta= \cI_1(y) + \cI_2(y).
\end{aligned}
\end{equation*}
\bullpar{Proof of \ref{item:i}}
We begin by estimating $\cI_1(y)$. Since we integrate in $(0,\delta)$ and $\delta\in (0,\delta_0)$, we can write $F(\eta ,y)=\eta ^{-\alpha}E_1(\eta ,y)$ and directly estimate using Minkowsky inequality
\begin{equation*}
\l\Vert \cI_1 \r\Vert_{L^2_y(\R)} \leq \Vert E_1 \Vert_X\int_0^\delta \eta ^{-\alpha}\d \eta  = \frac{\Vert E_1 \Vert_X}{1-\alpha}\delta^{1-\alpha}.
\end{equation*}
On the other hand, since $F$ vanishes at infinity, integrating by parts we can write
\begin{align*}
\cI_2 &= \frac{1}{imt}\int_\delta^\infty \p_\eta    \l( \e^{im\eta t}\r) F(\eta ,y)\d \eta  \\
&= -\frac{1}{imt} \e^{im\delta t}F(\delta,y) - \frac{1}{imt}\int_\delta^{\delta_0} \e^{im\eta t}\p_\eta   F(\eta ,y) \d \eta - \frac{1}{imt}\int_{\delta_0}^\infty \e^{im\eta t}\p_\eta   F(\eta ,y) \d \eta .
\end{align*}
and we estimate 
\begin{equation*}
\l\Vert \cI_2\r\Vert_{L^2_y(\R)} \lesssim \frac{1}{mt}\delta^{-\alpha} \l( \Vert E_1 \Vert_X  + \Vert E_2 \Vert_X\r)  +\frac{1}{mt}\Vert \p_\eta   F \Vert_Y.
\end{equation*}
Therefore, we conclude that
\begin{equation*}
\Vert \cI \Vert_{L^2_y(\R)} \lesssim  \l(\delta^{1-\alpha} + \frac{1}{mt}\delta^{-\alpha}\r)  \l( \Vert E_1 \Vert_X + \Vert E_2 \Vert_X\r) + \frac{1}{mt}\Vert \p_\eta    F \Vert_Y.
\end{equation*}
For $\delta=\frac{1}{4mt}<{\delta_0}$  we obtain the desired decay estimate
\begin{equation*}
\l\Vert \int_0^\infty \e^{im\eta t}F(\eta ,y)\d \eta \r\Vert_{L^2_y(\R)} \lesssim  \frac{1}{(mt)^{1-\alpha}} \l( \Vert E_1 \Vert_X + \Vert E_2 \Vert_X\r)  + \frac{1}{mt}\Vert \p_\eta   F \Vert_Y.
\end{equation*}
\bullpar{Proof of \ref{item:ii}} 
For $\cI_1$, since we have the expansion $F(\eta ,y)=\eta ^{-\alpha}(E_{1,1}(\eta ,y)+\log(\eta )E_{1,2}(\eta ,y))$ for $\eta \in(0,{\delta_0})$, we have that since $\delta<1$, 
\begin{equation*}
\begin{aligned}
\Vert \cI_1 \Vert_{L^2_y(\R)} &\lesssim \int_0^\delta \eta ^{-\alpha}\l(1 + |\log(\eta )|\r)\l(\Vert E_{1,1}\Vert_X + \Vert E_{1,2}\Vert_X\r) \d \eta  \\
&\lesssim \delta^{1-\alpha}\l( 1 + \big| \log\l(\delta\r) \big|\r) \l(\Vert E_{1,1}\Vert_X + \Vert E_{1,2}\Vert_X\r) .
\end{aligned}
\end{equation*}
As for $\cI_2$, integrating by parts, since $F$ vanishes at infinity and using the asymptotic expansion $\p_\eta   F(\eta ,y)=\eta ^{-\alpha-1}(E_{2,1}(\eta ,y) + \log(\eta )E_{2,2}(\eta ,y))$, one can estimate 
\begin{equation*}
\begin{aligned}
\Vert \cI_2 \Vert_{L^2_y(\R)} &\leq \frac{1}{mt}\l(  \Vert F(\delta,\cdot)\Vert_{L^2_y(\R)}+\int_\delta^{\delta_0} \Vert \p_\eta   F(\eta ,\cdot)\Vert_{L^2_y(\R)} \d \eta  + \int_{\delta_0}^\infty \Vert \p_\eta   F(\eta ,\cdot)\Vert_{L^2_y(\R)}\d \eta  \r) \\
&\lesssim \frac{1}{mt} \delta^{-\alpha} \l( 1 + \big| \log \l(\delta\r) \big| \r)  \sum_{i,j\in \l\lbrace 1,2\r\rbrace} \Vert E_{i,j} \Vert_X + \frac{1}{mt}\Vert \p_\eta   F\Vert_Y.
\end{aligned}
\end{equation*}
Choosing once again $\delta=\frac{1}{4mt}$ yields the estimate
\begin{equation*}
\l\Vert \int_0^\infty \e^{im\eta t}F(\eta ,y)\d \eta \r\Vert_{L^2_y(\R)} \lesssim \frac{1}{(mt)^{1-\alpha}} (1+\log(mt)) \sum_{i,j\in \l\lbrace 1,2\r\rbrace} \Vert E_{i,j} \Vert_X + \frac{1}{mt}\Vert \p_\eta   F\Vert_Y,
\end{equation*}
which concludes the proof.
\end{proof}
We now obtain the decay estimates for the stream-function $\psi_m$.

\begin{proposition}\label{prop: invdamp psi}
The following holds for all $t\geq 1$.
\begin{itemize}
\item If $\b^2\neq 1/4$, then
\begin{equation*}
\Vert \psi_m(t)\Vert_{L^2_y(\R)} \lesssim m^{-3}t^{-\frac32+\mu }  Q_{2,m} .
\end{equation*}
\item If $\b^2 = 1/4$, then
\begin{equation*}
\Vert \psi_m(t)\Vert_{L^2_y(\R)} \lesssim  m^{-3}t^{-\frac32}(1+\log \l(mt\r))  Q_{2,m}.
\end{equation*}
\end{itemize}
\end{proposition} 

\begin{proof}
We have from Theorem \ref{thm: main solution formula} that
\begin{equation*}
\begin{aligned}
\psi_m(t,y)&=\frac{\e^{-imyt}}{2m\pi}\cos(\gamma\pi) \l(\int_0^\infty \e^{im\eta t}W(\eta ) \int_0^\infty W(\xi)G_m(\eta,\xi,y)\d \xi\d \eta  \r. \\
&\qquad\qquad\qquad\qquad\qquad \l. -\int_0^\infty \e^{-im\eta t}W(\eta )\int_0^\infty W(\xi)G_m(-\eta,-\xi,y)\d \xi\d \eta  \r). 
\end{aligned}
\end{equation*}
We show the decay estimates for 
\begin{equation}\label{eq:Tpsi}
    \mathcal{T}^{+}(y):= \int_0^\infty \e^{im\eta t}W(\eta )\int_0^\infty W(\xi)G_m(\eta ,\xi, y)\d \xi\d \eta
\end{equation}
since one can directly replicate the arguments to obtain the same estimates for
\begin{equation*}
    \mathcal{T}^{-}(y):= \int_0^\infty \e^{-im\eta t}W(\eta )\int_0^\infty W(\xi)G_m(-\eta,-\xi,y)\d \xi\d \eta. 
\end{equation*}
The time decay is achieved appealing to Lemma \ref{lemma: generic decay-by-regularity}. Integrating \eqref{eq:Tpsi} by parts in $\eta$ provides
\begin{equation*}
\begin{aligned}
\mathcal{T}^+(y)&=\int_0^\infty \frac{1}{imt}\p_\eta   (\e^{im\eta t})W(\eta )\int_0^\infty W(\xi)G_m(\eta ,\xi, y)\d \xi\d \eta  \\
&=-\frac{1}{imt}\int_0^\infty \e^{im\eta t}\l( W'(\eta )\int_0^\infty W(\xi)G_m(\eta ,\xi, y)\d \xi + W(\eta )\int_0^\infty W(\xi)\p_\eta  G_m(\eta ,\xi, y)\d \xi \r)\d \eta . \\
\end{aligned}
\end{equation*}
and further define
\begin{equation}\label{eq: def F psi}
F(\eta ,y):=W'(\eta )\int_0^\infty W(\xi)G_m(\eta ,\xi, y)\d \xi + W(\eta )\int_0^\infty W(\xi)\p_\eta  G_m(\eta ,\xi, y)\d \xi.
\end{equation}
Clearly,
\begin{equation*}
\begin{aligned}
\p_\eta    F(\eta ,y)&=W''(\eta )\int_0^\infty W(\xi)G_m(\eta ,\xi, y)\d \xi + 2W'(\eta )\int_0^\infty W(\xi)\p_\eta  G_m(\eta ,\xi, y)\d \xi \\
&\quad+W(\eta )\int_0^\infty W(\xi)\p_\eta   ^2G_m(\eta ,\xi, y)\d \xi.
\end{aligned}
\end{equation*}
We begin by checking that $\p_\eta   F\in Y$. For this, \eqref{eq:Whittakereqn} yields 
\begin{equation*}
\int_{\delta_0}^\infty |W''(\eta )|\d \eta  = 4m^2 \int_{\delta_0}^\infty \l|-\frac{1}{4}+\b^2\frac{1}{(2m\eta)^2}\r| \l| W_{0,\gamma}(2m\eta)\r| \d \eta  \lesssim  m\Vert W_{0,\gamma} \Vert_{L^1}.
\end{equation*}
Similarly we easily estimate
\begin{equation*}
\int_{\delta_0}^\infty |W'(\eta )|\d \eta  \leq \Vert W' \Vert_{L^1}= \Vert W'_{0,\gamma}\Vert_{L^1}
\end{equation*}
and
\begin{equation*}
\int_{\delta_0}^\infty |W(\eta )|\d \eta  \leq \frac{1}{2m}\Vert W_{0,\gamma}\Vert_{L^1(0,\infty)}.
\end{equation*}
Moreover we have that,
\begin{equation}\label{bound int WdjG}
\begin{aligned}
\l\Vert \int_0^\infty W(\xi)\p_\eta   ^jG_m(\eta ,\xi, y)\d \xi\r\Vert_{L^2_y(\R)} &\leq \Vert W \Vert_{L^1} Q_{j,m} \leq \frac{1}{m}\Vert W_{0,\gamma} \Vert_{L^2} Q_{j,m},
\end{aligned}
\end{equation}
for all $j\geq0$.
With this, we infer that $\p_\eta   F\in Y$ and
\begin{equation*}
\Vert \p_\eta    F \Vert_Y \lesssim Q_{2,m}.
\end{equation*}
Next, we check the asymptotic expansions of $F(\eta ,y)$ and $\p_\eta   F(\eta ,y)$ for $\eta \in [0,{\delta_0}]$. For this, we will distinguish the two cases.

\bullpar{Case $\b^2\neq 1/4$}
We can use Lemma \ref{asymptotic expansion W} to write
\begin{equation*}
W(\eta )=\eta ^{\frac12-\mu} \mathcal{E}_{m,2}(\eta ),\qquad W'(\eta )=\eta ^{-\frac12-\mu} \mathcal{E}_{m,1}(\eta ), \qquad W''(\eta )=\eta ^{-\frac32-\mu} \mathcal{E}_{m,2}(\eta ),
\end{equation*}
which yields
\begin{equation*}
\begin{aligned}
F(\eta ,y)&=\eta ^{-\frac12-\mu}\mathcal{E}_{m,1}(\eta )\int_0^\infty W(\xi)G_m(\eta ,\xi, y)\d \xi  +\eta ^{\frac12+\mu}\mathcal{E}_{m,0}(\eta )\int_0^\infty W(\xi)\p_\eta  G_m(\eta ,\xi, y)\d \xi  \\
&=\eta ^{-\frac12 -\mu}E_1(\eta,y ),
\end{aligned}
\end{equation*}
with $\Vert E_1 \Vert_X\lesssim m^{-\frac12-\mu} Q_{1,m}$ and 
\begin{equation*}
\begin{aligned}
\p_\eta    F(\eta ,y)&=\eta ^{-\frac32-\mu}\mathcal{E}_{m,2}(\eta )\int_0^\infty W(\xi)G_m(\eta ,\xi, y)\d \xi + 2u^{-\frac12-\mu}\mathcal{E}_{m,1}(\eta )\int_0^\infty W(\xi)\p_\eta  G_m(\eta ,\xi, y)\d \xi \\
&\qquad\qquad+ \eta ^{\frac12-\mu}\mathcal{E}_{m,0}(\eta )\int_0^\infty W(\xi)\p_\eta   ^2G_m(\eta ,\xi,y )\d \xi  \\
&=\eta ^{-\frac32-\mu}E_2(\eta, y),
\end{aligned}
\end{equation*}
with $\Vert E_2 \Vert_X \lesssim m^{-\frac12-\mu} Q_{2,m}$. With this, for $\alpha=\frac{1}{2}+\mu$,  we show that $F(\eta ,y)$ defined above satisfies the conditions of Lemma \ref{lemma: generic decay-by-regularity} and we conclude that
\begin{equation*}
\begin{aligned}
\l\Vert \int_0^\infty \e^{im\eta t}F(\eta ,y) \d \eta \r\Vert_{L^2_y(\R)} &\lesssim \frac{1}{mt^{1-\alpha}}  Q_{2,m},
\end{aligned}
\end{equation*}
which yields the claimed bound for $\Vert \psi_m(t)\Vert_{L^2_y(\R)}$.

\bullpar{Case $\b^2=1/4$} We shall now use the asymptotic expansions of Lemma \ref{asymptotic expansion special W}. We will use these expansions to check the validity of the hypothesis required to apply Lemma \ref{lemma: generic decay-by-regularity}. In this direction, for $\eta \in (0,{\delta_0})$  we can write
\begin{equation*}
\begin{aligned}
F(\eta ,y)&=\eta ^{-\frac12}\l[\mathcal{E}_{m,1,1}(\eta )+\log(\eta )\mathcal{E}_{m,1,2}(\eta ) \r]\int_0^\infty W(\xi)G_m(\eta ,\xi, y)\d \xi  \\
&\quad+\eta ^{\frac12}\l[\mathcal{E}_{m,0,1}(\eta )+\log(\eta )\mathcal{E}_{m,0,2}(\eta )\r]\int_0^\infty W(\xi)\p_\eta  G_m(\eta ,\xi, y)\d \xi  \\
&=\eta ^{-\frac12}(E_{1,1}(\eta ,y)+\log(\eta )E_{1,2}(\eta ,y)),
\end{aligned}
\end{equation*}
with the uniform bounds
\begin{align*}
\Vert E_{1,1}\Vert_X\lesssim m^{-\frac12}\l( 1 + \log \l(m\r) \r) Q_{1,m}, \quad
\Vert E_{1,2}\Vert_X \lesssim m^{-\frac12} Q_{1,m}.
\end{align*}
Similarly, for $\p_\eta   F(\eta ,y)$ we can write
\begin{equation*}
\begin{aligned}
\p_\eta   F(\eta ,y)&=\eta ^{-\frac32}\l[\mathcal{E}_{m,2,1}(\eta ) +\log(\eta ) \mathcal{E}_{m,2,2}(\eta )\r]\int_0^\infty W(\xi)G_m(\eta ,\xi, y)\d \xi \\
&\quad+ 2u^{-\frac12}\l[\mathcal{E}_{m,1,1}(\eta )+\log(\eta )\mathcal{E}_{m,1,2}(\eta ) \r]\int_0^\infty W(\xi)\p_\eta  G_m(\eta ,\xi, y)\d \xi \\
&\quad+ \eta ^{\frac12}\l[\mathcal{E}_{m,0,1}(\eta )+\log(\eta )\mathcal{E}_{m,0,2}(\eta )\r]\int_0^\infty W(\xi)\p_\eta   ^2G_m(\eta ,v,\eta )\d \xi  \\
&=\eta ^{-\frac32}\big(E_{2,1}(\eta ,y)+\log(\eta )E_{2,2}(\eta ,y)\big),
\end{aligned}
\end{equation*}
with the bounds 
\begin{align*}
\Vert E_{2,1}\Vert_X &\lesssim m^{-\frac12}\l( 1 + \log \l(m\r) \r) Q_{2,m}, \quad \Vert E_{2,2}\Vert_X \lesssim m^{-\frac12} Q_{2,m}.
\end{align*}
Hence, we apply Lemma \ref{lemma: generic decay-by-regularity} for $\alpha=\frac12$ and $\delta=\frac{1}{4mt}$ to obtain
\begin{equation*}
\begin{aligned}
\l\Vert  \int_0^\infty \e^{im\eta t}F(\eta ,y)\d \eta \r\Vert_{L^2_y(\R)} &\lesssim \frac{1}{(mt)^{\frac12}}(1+\log \l(mt\r) )  \sum_{i,j\in \l\lbrace 1,2\r\rbrace} \Vert E_{i,j} \Vert_X + \frac{1}{mt}\Vert \p_\eta    F\Vert_Y\\ 
&\lesssim \frac{1}{mt^{\frac12}}(1+\log \l(mt\r) ) Q_{2,m}.
\end{aligned}
\end{equation*}
From here, the stated bound for $\Vert \psi_m(t)\Vert_{L^2_y(\R)}$ follows easily.
\end{proof}
From the explicit expression of $\psi_m(t,y)$ and replicating the proof of Proposition \ref{prop: invdamp psi}, one obtains the following result.

\begin{corollary}\label{cor: invdamp dy psi}
The following holds for all $t\geq 1$.
\begin{itemize}
\item If $\b^2\neq 1/4$, then
\begin{equation*}
\Vert \p_y\psi_m(t)\Vert_{L^2_y(\R)}\lesssim  m^{-2}t^{-\frac12+\mu} Q_{1,m}.
\end{equation*}
\item If $\b^2 = 1/4$, then
\begin{equation*}
\Vert \p_y \psi_m(t)\Vert \lesssim  m^{-2}t^{-\frac12}\l( 1 + \log \l(mt\r) \r) Q_{1,m}.
\end{equation*}
\end{itemize}
\end{corollary}

\begin{proof}
Note that 
\begin{align*}
\p_y\psi_m(t,y) &= -imt\psi_m(t,y) + \frac{\e^{-imyt}}{2m\pi}\cos(\gamma\pi) \l(\int_0^\infty \e^{im\eta t}W(\eta ) \int_0^\infty W(\xi)\p_yG_m(\eta,\xi,y)\d \xi\d \eta  \r. \\
&\qquad\qquad\qquad\qquad\qquad \l. -\int_0^\infty \e^{-im\eta t}W(\eta )\int_0^\infty W(\xi)\p_yG_m(-\eta,-\xi,y)\d \xi\d \eta  \r). 
\end{align*}
In particular, we observe that
\begin{align*}
-imt &\frac{\e^{-imyt}}{2m\pi}\cos (\mu\pi)\int_0^\infty \e^{im\eta t}W(\eta )\int_0^\infty W(\xi)G_m(\eta ,\xi, y) \d \xi \d \eta \\
&= -\frac{\e^{-imyt}}{2m\pi}\cos (\mu\pi)\int_0^\infty \p_\eta   \l(\e^{im\eta t}\r)W(\eta )\int_0^\infty W(\xi)G_m(\eta ,\xi, y) \d \xi \d \eta  \\
&= \frac{\e^{-imyt}}{2m\pi}\cos (\mu\pi)\int_0^\infty \e^{im\eta t}W'(\eta )\int_0^\infty W(\xi)G_m(\eta ,\xi, y) \d \xi \d \eta  \\
&\quad + \frac{\e^{-imyt}}{2m\pi}\cos (\mu\pi)\int_0^\infty \e^{im\eta t}W(\eta )\int_0^\infty W(\xi)\p_\eta    G_m(\eta ,\xi, y) \d \xi \d \eta .
\end{align*}
Under the observation that $(\p_\eta    + \p_y)\G_m(\eta ,\xi, y)=0$, we conclude that
\begin{align*}
\p_y\psi_m(t,y) &=\frac{\e^{-imyt}}{2m\pi}\cos(\gamma\pi) \l(\int_0^\infty \e^{im\eta t}W'(\eta ) \int_0^\infty W(\xi)\p_yG_m(\eta,\xi,y)\d \xi\d \eta  \r. \\
&\qquad\qquad\qquad\qquad\qquad \l. -\int_0^\infty \e^{-im\eta t}W'(\eta )\int_0^\infty W(\xi)\p_yG_m(-\eta,-\xi,y)\d \xi\d \eta  \r). 
\end{align*}
and the corollary follows applying Lemma \ref{lemma: generic decay-by-regularity}, we omit the details.
\end{proof}

We now obtain the decay in time of the perturbed density.
\begin{proposition}
The following holds for all $t\geq 1$.
\begin{itemize}
\item If $\b^2\neq 1/4$, then
\begin{equation*}
\Vert \rho_m(t)\Vert_{L^2_y(\R)}\lesssim  m^{-1}t^{-\frac12+\mu} Q_{1,m}.
\end{equation*}
\item If $\b^2 = 1/4$, then
\begin{equation*}
\Vert \rho_m(t)\Vert_{L^2_y(\R)}\lesssim  m^{-2}t^{-\frac12}\l( 1 + \log \l(mt\r) \r) Q_{1,m}.
\end{equation*}
\end{itemize}
\end{proposition}

\begin{proof}
From Theorem \ref{thm: main solution formula},
\begin{equation*}
\begin{aligned}
\rho_m(t,y)&=\frac{\e^{-imyt}}{2m\pi}\cos(\gamma\pi) \l(\int_0^\infty \e^{im\eta t}\frac{W(\eta )}{\eta } \int_0^\infty W(\xi)G_m(\eta,\xi,y)\d \xi\d \eta  \r. \\
&\qquad\qquad\qquad\qquad\qquad \l. +\int_0^\infty \e^{-im\eta t}\frac{W(\eta )}{\eta }\int_0^\infty W(\xi)G_m(-\eta,-\xi,y)\d \xi\d \eta  \r). 
\end{aligned}
\end{equation*}
As before, we only show the decay estimate for $$
\mathcal{T}:=\int_0^\infty \e^{im\eta t}\frac{W(\eta )}{\eta } \int_0^\infty W(\xi)G_m(\eta,\xi,y)\d \xi\d \eta. 
$$ Denoting
\begin{equation*}
F(\eta ,y)=\frac{W(\eta )}{\eta }\int_0^\infty W(\xi)G_m(\eta ,\xi, y)\d \xi \d \eta ,
\end{equation*} 
we shall apply Lemma \ref{lemma: generic decay-by-regularity}. We compute 
\begin{equation*}
\p_\eta F(\eta ,y)=\l(\frac{W'(\eta )}{\eta }-\frac{W(\eta )}{\eta ^2}\r)\int_0^\infty W(\xi)G_m(\eta ,\xi, y)\d \xi + \frac{W(\eta )}{\eta }\int_0^\infty W(\xi)\p_\eta  G_m(\eta ,\xi, y) \d \xi
\end{equation*}
and we observe the following bounds:
\begin{align*}
\int_{\delta_0}^\infty \l|\frac{W'(\eta )}{\eta }\r|\d \eta  &\leq {\delta_0}^{-1}\Vert W' \Vert_{L^1({\delta_0},\infty)}={\delta_0}^{-1}\Vert W_{0,\gamma}' \Vert_{L^1(1,\infty)} \\
\int_{\delta_0}^\infty \l|\frac{W(\eta )}{\eta ^2}\r|\d \eta  &\leq {\delta_0}^{-2}\Vert W \Vert_{L^1({\delta_0},\infty)}={\delta_0}^{-1}\Vert W_{0,\gamma} \Vert_{L^1(1,\infty)} \\
\int_{\delta_0}^\infty \l|\frac{W(\eta )}{\eta }\r|\d \eta  &\leq {\delta_0}^{-1}\Vert W \Vert_{L^1({\delta_0},\infty)}=\Vert W_{0,\gamma} \Vert_{L^1(1,\infty)} 
\end{align*}
since ${\delta_0}^{-1}=2m$. Together with \eqref{bound int WdjG} we deduce that $\p_\eta    F\in Y$ and we can estimate 
\begin{equation*}
\Vert \p_\eta    F \Vert_Y \lesssim Q_{1,m}.
\end{equation*} 
We next treat each $\b^2$ case separately to obtain the correct asymptotic expansions.

\bullpar{Case $\b^2\neq1/4$} Following the asymptotic expansions of Lemma \ref{asymptotic expansion W}, we can write
\begin{equation*}
F(\eta ,y)=\eta ^{-\frac12-\mu}\mathcal{E}_{m,0}(\eta )\int_0^\infty W(\xi)G_m(\eta ,\xi, y)\d \xi=\eta ^{-\frac12-\mu}E_{1}(\eta ,y),
\end{equation*}
with $\Vert E_1 \Vert_X \lesssim m^{-\frac12-\mu} Q_{0,m}$. Similarly, we have that
\begin{equation*}
\begin{aligned}
\p_\eta    F(\eta ,y)&=\l( \eta ^{-\frac32-\mu}\mathcal{E}_{m,1}(\eta )-\eta ^{-\frac32-\mu}\mathcal{E}_{m,0}(\eta )\r)\int_0^\infty W(\xi)G_m(\eta ,\xi, y)\d \xi \\
&\qquad+ \eta ^{-\frac12}\mathcal{E}_{m,0}(\eta )\int_0^\infty W(\xi)\p_\eta  G_m(\eta ,\xi, y)\d \xi \\
&= \eta ^{-\frac32-\mu}E_2(\eta ,y),
\end{aligned}
\end{equation*}
where $\Vert E_2 \Vert_X \lesssim m^{-\frac12-\mu} Q_{1,m}$. Hence, taking $\alpha=\frac12+\mu$ we apply Lemma \ref{lemma: generic decay-by-regularity} swiftly and obtain the decay estimate
\begin{equation*}
\l\Vert  \int_0^\infty \e^{im\eta t}F(\eta ,y)\d \eta \r\Vert_{L^2_y(\R)} \lesssim  m^{-1}t^{-\frac12+\mu} Q_{1,m},
\end{equation*}
from which the proof follows.

\bullpar{Case $\b^2=1/4$} Thanks to the asymptotic expansions of Lemma \ref{asymptotic expansion W}, we have
\begin{equation*}
\begin{aligned}
F(\eta ,y)&=\eta ^{-\frac12}\big(\mathcal{E}_{m,0,1}(\eta )+\log(\eta )\mathcal{E}_{m,0,2}(\eta )\big)\int_0^\infty W(\xi)G_m(\eta ,\xi, y)\d \xi \\
&=\eta ^{-\frac12}\big( E_{1,1}(\eta ,y)+\log(\eta )E_{1,2}(\eta ,y)\big),
\end{aligned}
\end{equation*}
with the bounds
\begin{align*}
\Vert E_{1,1}\Vert_X &\lesssim m^{-\frac12}\l( 1 + \log \l(m\r) \r) Q_{0,m}, \quad \Vert E_{1,2}\Vert_X \lesssim m^{-\frac12}  Q_{0,m}.
\end{align*}
As for $\p_\eta    F$, we have that
\begin{equation*}
\begin{aligned}
\p_\eta   F(\eta ,y)&=  \eta ^{-\frac32}\big(\mathcal{E}_{m,1,1}(\eta ) + \log(\eta )\mathcal{E}_{m,1,2}(\eta )\big)\int_0^\infty W(\xi)G_m(\eta ,\xi, y)\d \xi \\
 &\quad- \eta ^{-\frac32}\big(\mathcal{E}_{m,0,1}(\eta )+\log(\eta )\mathcal{E}_{m,0,2}(\eta )\big)\int_0^\infty W(\xi)G_m(\eta ,\xi, y)\d \xi \\
 &\quad+ \eta ^{-\frac12}\big(\mathcal{E}_{m,0,1}(\eta )+\log(\eta )\mathcal{E}_{m,0,2}(\eta )\big)\int_0^\infty W(\xi)\p_\eta  G_m(\eta ,\xi, y)\d \xi \\
&=\eta ^{-\frac32}\big( E_{2,1}(\eta ,y) + \log(\eta )E_{2,2}(\eta ,y)\big),
\end{aligned}
\end{equation*}
where we can bound
\begin{align*}
\Vert E_{2,1}\Vert_X &\lesssim m^{-\frac12} \l( 1 + \log \l(m\r) \r) Q_{1,m}, \quad \Vert E_{2,2}\Vert_X \lesssim  m^{-\frac12} Q_{1,m}.
\end{align*}
Now, for $\alpha=1/2$, we have that
\begin{equation*}
\l\Vert \int_0^\infty \e^{im\eta t}F(\eta ,y) \d \eta  \r\Vert_{L^2_y(\R)} \lesssim  m^{-1}t^{-\frac12}\l( 1 + \log \l(mt\r) \r) Q_{1,m}
\end{equation*}
due to Lemma \ref{lemma: generic decay-by-regularity}. With this, the proof is complete.
\end{proof}

\appendix

\section{The Whittaker functions}\label{app: Whittaker}
Here we give a description of the Whittaker function $W_{0,\gamma}$ and its asymptotic expansions. For $\mu=\Re\l( \sqrt{1/4 -\b^2}\r)$ and $\nu=\Im\l( \sqrt{1/4 - \b^2} \r)$ we set $\gamma=\mu + i\nu$. For $\gamma\neq 0$ and $\zeta\in\C$, the Whittaker function $W_{0,\gamma}(\zeta)$ is given by
\begin{equation}\label{eq:WtoM}
\begin{split}
W_{0,\gamma}(\zeta)=\frac{\Gamma(-2\gamma)}{\Gamma(\frac{1}{2}-\gamma)}M_{0,\gamma}(\zeta)+\frac{\Gamma(2\gamma)}{\Gamma(\frac{1}{2}+\gamma)}M_{0,-\gamma}(\zeta) 
\end{split}
\end{equation}
Here, $\Gamma(\cdot)$ stands for the Gamma function and the Whittaker functions $M_{0,\gamma}$ and $M_{0,-\gamma}$ are given by

\begin{equation*}\label{eq: def Mgamma}
M_{0,\pm\gamma}(\zeta) = \e^{-\frac12\zeta}\zeta^{\frac12\pm \gamma}M\l( \tfrac12\pm \gamma, 1\pm 2\gamma,\zeta\r), \quad M(a,b,\zeta)=\sum_{s=0}^\infty\frac{(a)_s}{(b)_s s!}\zeta^s,
\end{equation*}
where $(a)_s=a(a+1)(a+2)\dots (a+s-1)$. See \cite{CZN23chan, NIST} for more details.

The asymptotic estimates for $W_{0,\gamma}$ are deduced from the asymptotic estimates for $M_{0,\pm\gamma}$, recorded in Lemma A.3 from \cite{CZN23chan}, due to the relation \eqref{eq:WtoM}.
\begin{lemma}\label{asymptotic expansion W}
Let $\zeta\in\C$. Let $B_R\subset\C$ denote the closed unit ball of radius $R>0$ centered in the origin. Then,
\begin{equation*}
\begin{aligned}
W_{0,\gamma}(\zeta)&=\zeta^{\frac12 -\gamma}\mathcal{E}_{0,\gamma}(\zeta),\quad W_{0,\gamma}'(\zeta)=\zeta^{-\frac12 -\gamma}\mathcal{E}_{1,\gamma}(\zeta),
\end{aligned}
\end{equation*}
where $\mathcal{E}_{j,\gamma}\in L^\infty(B_R)$ and $\Vert \mathcal{E}_{j,\gamma}\Vert_{L^\infty(B_R)} \lesssim_{\gamma,R} 1$, for $j=0,1$.
\end{lemma}

For $\b^2=1/4$, we have $\gamma=0$ and \eqref{eq:WtoM} is no longer valid. Then, $W_{0,0}$ is given by 
\begin{equation}\label{eq:WtoK}
W_{0,0}(\zeta)=\sqrt{\frac{\zeta}{\pi}}K_0\l(\frac{\zeta}{2}\r),
\end{equation}
where $K_0$ is the modified Bessel function of second kind of order 0. See \cite{NIST} for more details on $K_0$. We next state the asymptotic expansions for $W_{0,0}$, which follow from \eqref{eq:WtoK} and are shown in \cite{CZN23chan}.

\begin{lemma}[\cite{CZN23chan}, Lemma A.4]\label{asymptotic expansion special W}
Let $\b^2=1/4$ and $\zeta\in\C$. Let $B_R\subset\C$ denote the closed ball of radius $R>0$ centered at the origin. Then,
\begin{equation*}
W_{0,0}(\zeta) = \zeta^\frac12 \big( \E_{0,1}(\zeta) - \log (\zeta) \E_{0,2}(\zeta)\big), \quad W_{0,0}'(\zeta) = \zeta^{-\frac12} \big( \E_{1,1}(\zeta) - \log (\zeta) \E_{1,2}(\zeta)\big),
\end{equation*}
where $\E_{j,k}(\zeta)$ are entire functions in $\C$ and $\Vert \E_{j,k}\Vert_{L^\infty(B_R)}\lesssim 1$, for $j=0,1$ and $k=1,2$. 
\end{lemma}

We finish our discussion with the proof of Lemma \ref{analytic continuation lemma} when $\b^2\neq 1/4$.
\begin{proof}[Proof of Lemma \ref{analytic continuation lemma}]
From \eqref{eq:WtoM} we write
\begin{equation*}
\begin{split}
W(\zeta)=A(\gamma)M_{0,\gamma}(2m\zeta) + B(\gamma)M_{0,-\gamma}(2m\zeta).
\end{split}
\end{equation*}
The analytic continuation property of $M_{0,\gamma}(\zeta)$, see \cite{NIST}, states that
\begin{equation*}
M_{0,\gamma}(\zeta \e^{\pm \pi i}) = \pm i \e^{\pm \gamma \pi i}M_{0,\gamma}(\zeta).
\end{equation*}
Therefore, we can write
\begin{equation*}
M_{0,\gamma}(2m(-\eta +i\ep))=M_{0,\gamma}(2m(\eta -i\ep)\e^{i\pi})=i\e^{\gamma\pi i}M_{0,\gamma}(2m(\eta -i\ep))
\end{equation*}
and
\begin{equation*}
M_{0,\gamma}(2m(-\eta -i\ep))=M_{0,\gamma}(2m(\eta +i\ep)\e^{-i\pi})=-i\e^{-\gamma\pi i}M_{0,\gamma}(2m(\eta +i\ep)).
\end{equation*}
Similarly, we have
\begin{equation*}
M_{0,-\gamma}(2m(-\eta +i\ep))=i\e^{-\gamma\pi i}M_{0,-\gamma}(2m(\eta -i\ep))
\end{equation*}
and
\begin{equation*}
M_{0,-\gamma}(2m(-\eta -i\ep))=-i\e^{\gamma\pi i}M_{0,-\gamma}(2m(\eta +i\ep)).
\end{equation*}
Now, we have that
\begin{equation*}
W(-\eta +i\ep)=i\e^{\gamma\pi i}A(\gamma)M_{0,\gamma}(2m(\eta -i\ep))+i\e^{-\gamma\pi i}B(\gamma)M_{0,-\gamma}(2m(\eta -i\ep))
\end{equation*}
and
\begin{equation*}
W(-\eta -i\ep)=-i\e^{-\gamma\pi i}A(\gamma)M_{0,\gamma}(2m(\eta +i\ep))-i\e^{\gamma\pi i}B(\gamma)M_{0,-\gamma}(2m(\eta +i\ep)).
\end{equation*}
Since both $M_{0,\gamma}(\zeta)$ and $M_{0,-\gamma}(\zeta)$ are continuous functions in the complex subset $\l\lbrace \zeta\in\C: \Re(\zeta)\geq0\r\rbrace,$ it is easily seen that
\begin{equation*}
\begin{split}
\lim_{\ep\rightarrow0}W(-\zeta +i\ep)-W(-\zeta -i\ep)&= i\l( \e^{\gamma\pi i}+ \e^{-\gamma\pi i}\r) \l(A(\gamma)M_{0,\gamma}(\zeta)+B(\gamma)M_{0,-\gamma}(\zeta)\r) \\
&=2i\cos(\gamma\pi)W(\zeta ).
\end{split}
\end{equation*}
\end{proof}

 \section*{Acknowledgments} 
The research of MCZ was partially supported by the Royal Society URF\textbackslash R1\textbackslash 191492 and EPSRC Horizon Europe Guarantee EP/X020886/1.

\bibliographystyle{abbrv}
\bibliography{CZN-InvBoussStripBiblio.bib}
\end{document}